\documentclass[reqno,11pt]{amsart}
\usepackage[latin1]{inputenc}
\usepackage{amsmath, amsthm, amssymb,mathrsfs,amsfonts}
\usepackage{array}
\usepackage{graphicx}
\usepackage[margin=3cm]{geometry}
\usepackage[spanish,portuguese,english]{babel}
\usepackage{enumerate}
\usepackage{multicol}
\newtheorem{teor}{Theorem}[section]

\newtheorem{lema}{Lemma}[section]

\newtheorem{cor}{Corollary}[section]

\def\t{\mathop{\rm t }\nolimits}
\def\tl{\mathop{\rm t^L}\nolimits}
\def\tnil{\mathop{\rm t_{nil}}\nolimits}
\def\cl{\mathop{\rm cl}\nolimits}

\title[Symmetric elements under oriented involutions]{Lie nilpotency indices of symmetric elements under oriented involutions in group algebras}
\author[J.H. Castillo]{John H. Castillo}
\address{John H. Castillo, Departamento de Matemáticas y Estadística, Universidad de Nariño}
\email{jhcastillo@udenar.edu.co--jhcastillo@gmail.com}
\begin{document}
\maketitle
 \noindent
\begin{abstract}
Let $G$ be a group and let $F$ be a field of characteristic
different from $2$. Denote by $(FG)^+$ the set of symmetric elements
and by $\mathcal{U}^+(FG)$ the set of symmetric units, under an
oriented classical involution of the group algebra $FG$. We give
some lower and upper bounds on the Lie nilpotency index of $(FG)^+$
and the nilpotency class of $\mathcal{U}^+(FG)$.
\end{abstract}

\section{Introduction}
Let $FG$ denote the group algebra of a group $G$ over a field $F$
with $char(F)=p\neq 2$. A homomorphism $\sigma:G\rightarrow \{ \pm
1\}$  is called an \textit{orientation} of the group $G$. Working in
the context of $K$-theory, Novikov \cite{NOV}, introduced  an
\textit{oriented involution} $*$ of $FG$, given by
$$\left(\sum_{g\in G} \alpha_gg\right)^{*}=\sum_{g\in G}\alpha_g\sigma(g)g^{-1}.$$
When $\sigma$ is trivial this involution coincides with the so
called \textit{classical involution} of $FG$.

We denote $(FG)^{+}=\{\alpha\in FG:\alpha^{*}=\alpha\}$ and
$(FG)^{-}=\{\alpha\in FG:\alpha^{*}=-\alpha\}$ the set of symmetric
and skew-symmetric elements of $FG$ under $*$, respectively. We
denote by $N$ the kernel of $\sigma$. It is obvious that the
involution $*$ coincides on the group algebra $FN$ with the
classical involution.
 It is easy to see that,
as an $F$-module, $(FG)^{+}$ is generated by the set
\begin{align*}
\mathcal{S}=\{g+g^{-1}:g\in N \}\cup\{g-g^{-1}:g\in G\setminus N,g^2\neq 1\}
\end{align*}
and $(FG)^{-}$ is generated by
\begin{align*}
\mathcal{L}=\{g+g^{-1}:g\in
G\setminus N\}\cup\{g-g^{-1}:g\in N,g^2\neq 1\}.
\end{align*}

Given $g_1, g_2\in G$, we define the commutator
$(g_1,g_2)=g_1^{-1}g_2^{-1}g_1g_2$ and recursively, \linebreak
$(g_1,\ldots,g_{n})=((g_1,\ldots,g_{n-1}),g_{n})$ for $n$ elements
$g_1,\ldots,g_n$ of $G$. By the commutator $(X,Y)$ of the subsets
$X$ and $Y$ of $G$ we mean the subgroup of $G$ generated by all
commutators $(x,y)$ with $x\in X$, $y\in Y$. In this way, we can
define the lower central series of a nonempty subset $H$ of $G$ by:
$\gamma_1(H)=H$ and $\gamma_{n+1}(H)=(\gamma_n(H),H)$, for $n\geq
1$. We say that $H$ is nilpotent if $\gamma_n(H)=1$, for some $n$.
For a nilpotent subset $H\subseteq G$ the number $\cl(H)=\min\{n\in
\mathbb{N}_0: \gamma_{n+1}(H)=1\}$ is called the nilpotency class of
$H$. It can be proved that $H$ is a nilpotent set if and only if $H$
satisfies the group identity $(g_1,\ldots,g_n)=1$ for some $n\geq
2$.

In an associative ring $R$,  the Lie bracket on two elements  $x,y
\in R$ is defined by $[x,y]=xy-yx$. This definition is extended
recursively via $[x_1,\ldots,x_{n+1}]=[[x_1,\ldots,x_n],x_{n+1}]$.
For $X,Y\subseteq R$ by $[X,Y]$ we denote the additive subgroup
generated by all Lie commutators $[x,y]$ with $x\in X, y\in Y$. The
lower Lie central series of a nonempty subset $S$ of $R$ is defined
inductively by setting $\gamma^1(S)=S$ and
$\gamma^{n+1}(S)=[\gamma^n(S),S]$. We say that the subset $S$ is Lie
nilpotent if there exists a natural number $n$, such that
$\gamma^n(S)=0$. The smallest natural number with the last property,
denoted by $t(S)$, is called the Lie nilpotency index of $S$. It is
possible to show that $S$ is Lie nilpotent if and only if $S$
satisfies the polynomial identity $[x_1,\ldots,x_n]=0$ for some
$n\geq 2$.

Given a nonempty subset $S$ of $R$, we let $S^{(1)}=R$, and then for
each $i\geq 2$, let $S^{(i)}$ be the (associative) ideal of $R$
generated by all elements of the form $[a,b]$, with $a\in
S^{(i-1)},b\in S$. We say that $S$ is strongly Lie nilpotent if
$S^{(i)}=0$ for some $i$. The minimal $n$ for which $S^{(n)}=0$ is
called the upper Lie nilpotency index and denoted by $\tl(S)$.
Clearly, strong Lie nilpotence implies Lie nilpotence and $\t(S)\leq
\tl(S)$. Denote by $\mathcal{U}(S)$ the set of units in the subset
$S$ of $R$ and suppose that it is nonempty. By the equality
$(x,y)=1+x^{-1}y^{-1}[x,y]$, it is easy to see that
$\gamma_n(\mathcal{U}(S))\subseteq 1+S^{(n)}$ for all $n\geq 2$. In
consequence, the set of units of a strongly Lie nilpotent subset $S$
is nilpotent, and
\begin{equation}\label{eq1}
\cl(\mathcal{U}(S))<\tl(S).
\end{equation}

 In 1973,   Passi, Passman and Sehgal \cite{PPS73} showed that
the group algebra $FG$ is Lie nilpotent if and only if $G$ is
nilpotent and $G'$ is a  finite $p$-group, where $p$ is the
characteristic of $F$. Actually, see \cite{S78}, a group algebra is
Lie nilpotent if and only if it is strongly Lie nilpotent. Next,
S.K. Sehgal characterized group algebras which are Lie $n$-Engel,
for some $n$.

In 1993,  Giambruno and Sehgal \cite{SG93} began the study of Lie
nilpotence of
 symmetric  and skew-symmetric elements  under the classical
involution. They proved that given a group $G$ without elements of
order $2$ and a field $F$ with $char(F)\neq 2$, if either $(FG)^+$
or $(FG)^-$  is Lie nilpotent, then $FG$ is Lie nilpotent. This work
was completed by G.T. Lee \cite{Lee99}, for groups in general. More
specifically, he proved that the Lie nilpotence of the symmetric
elements under the classical involution is equivalent to the Lie
nilpotence of $FG$ when the group $G$ does not contain a copy of
$Q_8$, the quaternion group of order $8$ and he also characterized
the group algebras such that the set of symmetric elements is Lie
nilpotent when $G$ contains a copy of $Q_8$.

Recently, Castillo and Polcino Milies, see \cite{CP12}, studied Lie
properties of the symmetric elements under an oriented classical
involution. They extended some previous results from \cite{SG93},
\cite{Lee99} and \cite{Lee00}. In particular, they gave some groups
algebras such that the Lie nilpotence of the symmetric set implies
the same property in the whole group algebra. Also, they obtained a
complete characterization of the group algebras $FG$, such that
$Q_8\subseteq G$ and $(FG)^+$ is Lie nilpotent.

Lately, Z. Balogh and T. Juhász in \cite{BJ11} and \cite{BJ12}
studied the Lie nilpotency index of $(FG)^+$ and the nilpotency
class of the $\mathcal{U}^+(FG)$ under the classical involution in
group algebras. They gave a necessary condition to the numbers
$\t((FG)^+)$ and $\cl(\mathcal{U}^+(FG))$ be maximal, as possible,
in a nilpotent group algebra. Also, they studied this two numbers to
group algebras such that $(FG)^+$ is Lie nilpotent but $FG$ is not.

In this article we study the Lie nilpotency index  of $(FG)^+$ and
the nilpotency class of $\mathcal{U}^+(FG)$ under an oriented
classical involution. In the next section we give some preliminary
results. In the third section we study the numbers $\t((FG)^+)$ and
$\cl(\mathcal{U}^+(FG))$ in Lie nilpotent group algebras. In the
fourth section we study the case when $Q_8\subseteq G$ and $(FG)^+$
is Lie nilpotent.

Throughout this paper $F$ will always denote a field of
characteristic not $2$, $G$ a group and $\sigma$ a nontrivial
orientation of $G$. In a number of places, all over this paper, we
use arguments from \cite{BJ11}, \cite{BJ12} and \cite{Lee10}. Some
of them are reproduced here for the sake of completeness.

\section{Preliminaries}

 We recall the following result from \cite{Lee10}.

\begin{lema}\label{lema2.1}
Let $R$ be a ring and $S$ a subset of $R$. Suppose, for some $i\geq
1$, that $S^{(i)}\subseteq zR$, where $z$ is central in $R$. Then
for all $j>0$, we have $S^{(i+j)}\subseteq zS^{(j)}$. In particular,
for any positive integer $m$, $S^{(mi)}\subseteq z^m R$.
\end{lema}
\begin{proof}
The proof is by induction on $j$. If $j=1$, then $S^{(i+1)}\subseteq
S^{(i)}$, there is nothing to do. Assume that $S^{(i+j)}\subseteq
zS^{(j)}$. Take $a\in S^{(i+j)}, b\in S$. So $a=za_1$, for some
$a_1\in S^{(j)}$. Thus, $[a,b]=[za_1,b]=z[a_1,b]\in zS^{(i+j)}$, as
we want to prove.

To get the second part, notice that
$$S^{(2i)}=S^{(i+i)}\subseteq zS^{(i)}\subseteq z^2R.$$
Suppose that $S^{((m-1)i)}\subseteq z^{m-1}R$. So
$S^{(mi)}=S^{((m-1)i+i)}\subseteq zS^{((m-1)i)}\subseteq z^m R$.
\end{proof}

Throughout this article we denote by $Q_8=\left<x,y:x^4=1,x^2=y^2,x^y=x^{-1}\right>$
the quaternion group of order $8$. Castillo and Polcino Milies \cite{CP12} characterized the group
algebras of groups containing $Q_8$ and with a nontrivial orientation, such that $(FG)^+$
is Lie nilpotent. Here we prove that the conditions obtained by them
are also satisfied when $(FG)^+$ is strongly Lie nilpotent.

\begin{teor}\label{teorliengelQ8+}
Let $F$ be a field of characteristic $p\neq2$, $G$ a group with a
nontrivial orientation $\sigma$ and $x,y$ elements of $G$ such that
$\left< x,y\right>\simeq Q_8$. Then $(FG)^{+}$ is strongly Lie
nilpotent if and only if either
\begin{enumerate}[(i)]
\item $char(F)=0$, $N\simeq Q_8\times E$ and $G\simeq \left<Q_8,g\right> \times E$, where  $E^2=1$ and $g\in
G\setminus N$ is such that $(g,x)=(g,y)=1$ and $g^2=x^2$; or,
\item $char(F)=p>2$, $N\simeq Q_8\times E\times P$, where $E^2=1$, $P$ is a finite $p$-group
 and there exists $g\in G\setminus N$ such that  $G\simeq
\left<Q_8,g\right>\times E\times P$, $(g,x)=(g,y)=1$ and $g^2=x^2$.
\end{enumerate}
 \end{teor}
 \begin{proof}
 If $(FG)^+$ is strongly Lie nilpotent, then $(FG)^+$ is Lie
 nilpotent and from \cite[Theorem~4.2]{CP12} we get (i) and (ii).

Conversely, assume that $|P|=p^n$. We claim that,
$((FG)^+)^{(2p^n)}=0$. The proof will be by induction on $n$. If
$n=0$, then $G\simeq \left<Q_8,g\right>\times E$ and thus, from \cite[Lemma 4.3]{CP12}, $(FG)^+$
is commutative. Assume that  $|P|=p^n>1$. Take $z\in \zeta(P)$ with
$o(z)=p$, applying our inductive hypothesis on
$\overline{G}=G/\left<z\right>$. Then,
$((F\overline{G})^+)^{(2p^{n-1})}=0$. Thus
$$((FG)^+)^{(2p^{n-1})}\subseteq \Delta(G,\left<z\right>)=(z-1)FG.$$
By Lema \ref{lema2.1},
$$((FG)^+)^{(2p^n)}\subseteq (z-1)^pFG=0,$$
as we claimed.

 \end{proof}
From the equality, $(x,y)=1+x^{-1}y^{-1}[x,y]$ we know that
$\gamma_n(\mathcal{U}^+(FG))\subseteq 1+((FG)^+)^{(n)}$ and thus we
get the following.

\begin{cor}
Let $F$ be a field of characteristic different from $2$. Assume that
$Q_8\subseteq G$ and $(FG)^+$ is Lie nilpotent. Then,
$\mathcal{U}^+(FG)$ is nilpotent.
\end{cor}

We need the following easy observation.

\begin{lema}\label{lema5.1.1}
Let $G$ be a group, $H$ any subgroup and $A$ a normal subgroup such
that $A\subseteq N$. If $(FG)^+$ is Lie nilpotent, then so are
$(FH)^+$ and $(F(G/A))^+$. Furthermore, $\t((FH)^+)\leq \t((FG)^+)$
and $\t((F(G/A))^+)\leq \t((FG)^+)$.
\end{lema}
\begin{proof}
Note that $(FH)^+$ is a subset of $(FG)^+$, and thus it has the
required properties.

Since $A$ is a normal subgroup contained in the kernel of the
orientation $\sigma$, we can define in $F(G/A)$ an induced oriented classical
involution from $*$ in $FG$ as follows:
$$\left(\sum_{\bar{g}\in G/A}\alpha_g\bar{g}\right)^{\star}= \sum_{\bar{g}\in G/A}\alpha_g\sigma(g)\bar{g}^{-1}.$$

Now, simply observe that the symmetric elements in $F(G/A)$, under
$\star$, are linear combinations  of terms of the form
$gA+\sigma(g)g^{-1}A$, with $g\in G$. That is, every element of
$(F(G/A))^+$ is the homomorphic image of an element of $(FG)^+$
under the natural map $\varepsilon_A:FG\rightarrow F(G/A)$, defined by $\varepsilon_A(\sum_{g\in G} \alpha_g g)=\sum_{g\in G} \alpha_g \bar{g}$.

So assume that $(FG)^+$ is Lie nilpotent, therefore there exists
$n=\t((FG)^+)$ such that  $[\alpha_1,\ldots,\alpha_n]=0$ for all
$\alpha_i\in (FG)^+$. Let $\beta_1,\ldots,\beta_n\in (F(G/A))^+$.
Thus
\begin{align*}
[\beta_1,\ldots,\beta_n]&=[\varepsilon_A(\alpha_1),\ldots,\varepsilon_A(\alpha_n)]\\
&=\varepsilon_A([\alpha_1,\ldots,\alpha_n])=\varepsilon_A(0)=0.
\end{align*}
Consequently, $\t((F(G/A))^+)\leq \t((FG)^+)$.
\end{proof}

\section{Lie nilpotent group algebras}

In this section we assume that $FG$ is Lie nilpotent. By
\cite{SHB92}, $\tl(FG)\leq |G'|+1$ and by \cite{BS04} the equality
holds if and only if $G'$ is cyclic, or $G'$ is a noncentral
elementary abelian group of order $4$.

Note that a group $G$ of odd finite order has trivial orientation.
Indeed, let $a$ be an element of $G$. So
$1=\sigma(a^{|G|})=\sigma(a)^{|G|}$ and as $|G|$ is odd we get that
$\sigma(a)=1$. For the last reason when $G$ is a group of odd finite
order, the involution $*$ is the classical involution. In this way,
we can use the following result, that is a combination from \cite[Lemma 2]{BJ11} and \cite[Lemma 2]{BJ12}.

\begin{lema}\label{lema2Juhasz}
Let $G$ be a finite $p$-group with a cyclic derived subgroup. Then
$\t((FG)^+)\geq |G'|+1$ and $\cl(\mathcal{U}^+(FG))\geq |G'|$.
\end{lema}

We recall that a group $G$ is called $p$-abelian if $G'$, the
commutator subgroup of $G$, is a finite $p$-group and $0$-abelian
means abelian.

\begin{teor}
Let $FG$ be a Lie nilpotent group algebra of odd characteristic and
nontrivial orientation. Then, $\t((FG)^+)=|G'|+1$ if and only if
$G'$ is cyclic. Moreover, assuming that $G$ is a torsion group,
$\cl(\mathcal{U}^+(FG))=|G'|$ if and only if $G'$ is cyclic.
\end{teor}

\begin{proof} Assume that $\t((FG)^+)=|G'|+1$. As $G'$ is a finite $p$-group, if $G'$ is not
cyclic, from \cite{BS04}, we know that $t((FG)^+)\leq
\tl(FG)<|G'|+1$ and we get a contradiction. Thus, $G'$ is cyclic.

Conversely, suppose that $G'$ is cyclic. By the hypotheses, $G$ is a
nilpotent $p$-abelian group and from \cite[Lemma 1]{BJ08} there
exists a finite $p$-group $P$ which is isomorphic to a subgroup of
factor group of $G$ and $P'\simeq G'$. Actually, from the proof of
\cite[Lemma 1]{BJ08}, we know that $P\simeq H/A$, where $A$ is a
maximal torsion-free central subgroup of $G$.

Assume that there exists $g\in A$ such that $\sigma(g)=-1$. In this
way, as $G=N\cup gN$, we get $G'=N'$. Using in $FP$ the classical
involution, by lemmas \ref{lema2Juhasz} and \ref{lema5.1.1}, we obtain
that
$$|G'|+1=|N'|+1=|P'|+1\leq \t((FP)^+)\leq \t((FN)^+)\leq \t((FG)^+).$$
In the other hand, suppose that $A\subseteq N$. Then we can define
an induced oriented classical involution in $P\simeq H/A$, from that
one in $FG$. Consequently,
$$|G'|+1=|P'|+1\leq \t((FP)^+)\leq \t((FG)^+).$$

The proof of the second part is similar.
\end{proof}

\section{Groups that contain a copy of $Q_8$}
We assume that $Q_8\subseteq G$ and $(FG)^+$ is Lie nilpotent. This
means that the group algebra $FG$ is not Lie nilpotent. Recently,
this kind of group algebras was characterized by Castillo and
Polcino Milies \cite{CP12}. This characterization is the same as in
Theorem \ref{teorliengelQ8+}, so during this section we assume that
$G$ is as in that result. In this section, we will study the Lie
nilpotency index of the symmetric elements under oriented classical
involutions.

It is easy to show that
\begin{equation}\label{eq5.5}
g^m-1\equiv m(g-1)\pmod {\Delta(G)^2}.
\end{equation}
for every $g\in G$ and any integer $m$.

We begin with the following result.
\begin{lema}\label{lema5Juhasz}
Consider $FG$ with an oriented classical involution. Then
$$((FG)^+)^{(n)}\subseteq FG\Delta(P)^n$$ for all $n\geq 2$
\end{lema}
\begin{proof}

Recall that the symmetric elements are spanned as an $F$-module by
the set $$\mathcal{S}=\{z+z^{-1}: z\in N\}\cup \{z-z^{-1}:z\in
G\setminus N\}.$$ If $z\in N$, then $z=ah$ with $a\in Q_8\times E$
and $h\in P$. Note that if $a^2h=1$, then $h=1$ and $a^2=1$. Thus,
$a\in \zeta(Q_8\times E)$. Assuming $a^2h\neq 1$, follows that
$z+z^{-1}=ah+a^{-1}h^{-1}=ah+a^3h^{-1}=a(h+a^2h^{-1})$.

Also, if $z\in G\setminus N$; we can write $z=gah$ with $a\in
Q_8\times E$ and $h\in P$. If $a^2h=1$, then $a^2=h=1$. Again, $a\in
\zeta(Q_8\times E)$ and thus
$z-z^{-1}=gah-g^{-1}a^{-1}h^{-1}=ga-g^{-1}a=ga(1-g^2)\in
\zeta(Q_8\times E)$. Now we suppose that $a^2h\neq 1$ and we get the
following cases:
\begin{enumerate}
\item If $a^2=1$ and $h\neq 1$, then
$z-z^{-1}=gah-g^{-1}a^{-1}h^{-1}=ag(h-g^2h^{-1})$.
\item If $a^2\neq 1$ and $h=1$, then
$z-z^{-1}=gah-g^{-1}a^{-1}h^{-1}=ga-g^3a^3=ga-ga=0$.
\item If $a^2\neq 1$ and $h\neq 1$, then $z-z^{-1}=gah-g^{-1}a^{-1}h^{-1}=agh-a^3g^3h^{-1}=ag(h-h^{-1})$,
because $a^3g^3=ag$.
\end{enumerate}

 From the above considerations, we obtain that
\begin{align*}
 \mathcal{S}=\mathcal{A}\cup \mathcal{B}\cup\mathcal{C} \cup\zeta(Q_8\times
 E),
 \end{align*}
where
\begin{align*}
\mathcal{A}&=\{a(h+a^2h^{-1}): a\in Q_8\times E, h\in P \text{ and }
a^2h\neq 1 \},\\
\mathcal{B}&=\{ag(h-g^2h^{-1}):a\in Q_8\times E , h\in P \text{ and
}
(a^2=1 \text{ and } h\neq 1) \},\\
\mathcal{C}&=\{ag(h-h^{-1}):a\in Q_8\times E , h\in P \text{ and }
(a^2\neq 1 \text{ and } h\neq 1)\}.
\end{align*}

Given $a\in Q_8\times E$, such that $a^2\neq 1$ we know that $1+a^2$
is symmetric and $a^2\in \zeta(Q_8\times E)$. In this way,
$$a(h+a^2h^{-1})+1+a^2=a(h-1)+a^3(h^{-1}-1)+1+a+a^2+a^3,$$
where $1+a+a^2+a^3$ is a central element in $FG$ and
$a(h-1)+a^3(h^{-1}-1)\in FG\Delta(P)$. It is clear that,
$ag(h-h^{-1})\in FG\Delta(P)$. Furthermore, if $a^2=1$ and $h\neq
1$, then $ag(h-g^2h^{-1})=ag(h-1)-ag^3(h^{-1}-1)+a(g-g^{-1})\in
FG\Delta(P)+\zeta(FG)$.

So
\begin{align*}
 \mathcal{\widetilde{S}}=\mathcal{A}'\cup \mathcal{B}\cup \mathcal{C} \cup\zeta(Q_8\times
 E),
\end{align*}
also spans $(FG)^+$ as an $F$-module, where
$$\mathcal{A}'=\{a(h+a^2h^{-1})+1+a^2: a\in Q_8\times E, h\in P
\text{ and } a^2h\neq 1 \}$$  and $\mathcal{B}, \mathcal{C}$ are as
above.

In consequence,
\begin{equation}\label{contencion1} (FG)^+\subseteq
FG\Delta(P)+\zeta(FG).
\end{equation}
The proof follows by induction on $n$. Indeed, if $n=2$
\begin{align*}
[(FG)^+,(FG)^+]\subseteq [FG\Delta(P),FG\Delta(P)]\subseteq
FG\Delta(P)^2.
\end{align*}
Suppose that the lemma is true for some $n\geq 2$. Take $\alpha\in
((FG)^+)^{(n)}$ and $\beta \in (FG)^+$. So
\begin{align*}
[\alpha,\beta] \in [FG\Delta(P)^n,FG\Delta(P)]\subseteq
FG\Delta(P)^{n+1}.
\end{align*}
and we get that $((FG)^+)^{(n+1)}\subseteq FG\Delta(P)^{n+1}$ as
required.
\end{proof}

Denote by $c$ the central element of $Q_8\times E$, such that
$(Q_8\times E)^2=\left<c\right>$. Given $n\geq 2$, we denote with
$M_n$ the $F$-subspace of the vector space $FG$  generates by the
set
$$\{(h_1-h_1^{-1})\cdots (h_n-h_n^{-1})(1-c)a:h_1,\ldots,h_n\in P, a\in (Q_8\times E)\setminus \zeta(Q_8\times E)\}.$$
To simplify, we write $f_{1,\ldots,n}$ instead of
$(h_1-h_1^{-1})\cdots (h_n-h_n^{-1})$.

Let $S_n$ be the symmetric group of degree $n$ and $FS_n$ its group
algebra over the field $F$. It is possible to define a group action
of $S_n$ on $M_n$ via: for a $\sigma\in S_n$ and a generator element
$f_{1,\ldots,n}(1-c)a$ of $M_n$ let
$$
\sigma\cdot
f_{1,\ldots,n}(1-c)a=f_{\sigma(1),\ldots,\sigma(n)}(1-c)a.$$
Naturally, this group action on a generator set of $M_n$ can be
extended linearly to the whole $M_n$. We extend this group action to
a group algebra action: for $x=\sum_{\sigma\in
S_n}\alpha_{\sigma}\sigma \in FS_n$ and $z\in M_n$, let
$$x\cdot z=\sum_{\sigma\in S_n}\alpha_{\sigma}(\sigma\cdot z).$$
For $n\geq 2$ we define the elements
$x_{2,n},x_{3,n},\ldots,x_{n,n}$ of $FS_n$ recursively as:
\begin{align}
&x_{2,n}=1+(2,1),\\
&x_{i,n}=x_{i-1,n}+x_{i-1,n}(i,i-1,\ldots,1); \text{ for $3\leq
i\leq n$.}
\end{align}

Since $(FN)^+\subseteq (FG)^+$, from Lemma 4 and Lemma 5 in
\cite{BJ12}, we get the following results.
\begin{lema}
$x_{n,n}M_n\in \gamma^n((FG)^+)(1-c)$ for all $n\geq 2$.
\end{lema}

\begin{lema}\label{lema6Juhasz}
If $|P|=p^k$, then $\widehat{P}(1-c)a\in \gamma^{k(p-1)}((FG)^+)$
for some $a\in Q_8\times E$
\end{lema}

We recall that the augmentation ideal $\Delta(P)$ of a finite
$p$-group $P$ is a nilpotent ideal, see \cite[Theorem 6.3.1]{PS02},
we will denote by $\tnil(P)$ its nilpotency index. Also, we remind
that a finite $p$-group $P$, is called \emph{powerful} if
$P'\subseteq P ^p$.
 Let $P$  be a powerful group. We denote with $D_i=D_i(FP)$  the
$i$-th dimensional subgroup. By Theorem 5.5 in \cite{J41}, $D_1=P$
and for $n>1$,
$$D_n=\left<(D_{n-1},P),(D_{\lceil \frac{n}{p}\rceil})^p\right>.$$
It can be showed that, $(P^{p^i})^{p^j}=P^{p^{i+j}}$ and
$(P^{p^i},P)\subseteq P^{p^{i+1}}$ for every pair $i,j$. So, if
$p^{i-1}<n\leq p^i$ then $D_n=P^{p^i}$.
\begin{lema}\label{lema5.3.4}
Let $P$ be a powerful group and $h_i-1\in\Delta(P)^{k_i}$ and
$h_j-1\in \Delta(P)^{k_j}$, where $k_i$ and $k_j$ are positive
integers. Then
\begin{equation}\label{eq5.24}
(h_i-1)(h_j-1)\equiv (h_j-1)(h_i-1)\pmod{\Delta(P)^{k_i+k_j+1}}.
\end{equation}
\end{lema}
\begin{proof}
First, we prove that $(D_i,D_j)\subseteq D_{i+j+1}$, for every
$i,j$. Take $h_i\in D_i$ and $h_j\in D_j$. We get the following
equation
\begin{equation}\label{eq5.25}
(h_i,h_j)-1=h_i^{-1}h_j^{-1}((h_i-1)(h_j-1)-(h_j-1)(h_i-1)).
\end{equation}

If either $i$ or $j$, say $i$, is not a power of  $p$, then $h_i\in
D_i=D_{i+1}$, so by \eqref{eq5.25}, $(h_i,h_j)-1\in
\Delta(P)^{i+j+1}$; thus $(h_i,h_j)\in D_{i+j+1}$. If both $i$ and
$j$ are powers of $p$, then $i+j$ cannot be a power of $p$ and
consequently $D_{i+j}=D_{i+j+1}$. By \eqref{eq5.25} follows
$(h_i,h_j)\in D_{i+j+1}$; therefore our claim is proved.

Let $h_i-1\in \Delta(P)^{k_i}$ and $h_j-1\in \Delta(P)^{k_j}$ for
some positive integers $k_i,k_j$. Then
\begin{equation*}
(h_i-1)(h_j-1)=(h_j-1)(h_i-1)+h_jh_i((h_i,h_j)-1),
\end{equation*}
and as $(h_i,h_j)\in D_{k_i+k_j+1}$, the result follows.
\end{proof}

Now we can prove our main result in this section.
\begin{teor}
Let $F$ be a field of characteristic $p>2$. Consider
 the group algebra $FG$ with an oriented classical involution. Assume
that $Q_8\subseteq G$, $(FG)^+$ is Lie nilpotent and the Sylow
$p$-group $P$ of $G$ is of order $p^m$, with $m\geq 1$. Then
\begin{enumerate}[(i)]
\item $1+m(p-1)\leq \t((FG)^+)\leq \tl((FG)^+)\leq \tnil(P)$ and $\cl(\mathcal{U}^+(FG))\leq \tnil(P)-1$.
\item If $\t((FG)^+)=\tnil(P)$, then $\cl(\mathcal{U}^+(FG))+1=\t((FG)^+)$.
\item If $P$ is powerful, then $\t((FG)^+)=\tnil(P)$.
\item If $P$ is abelian, then, for all $k\geq 2$, the $F$-space
$\gamma^k((FG)^+)$ is generated by the set
\begin{equation*}
\begin{split}
\mathcal{M}_k=&\{(h_1-h_1^{-1})\cdots (h_k-h_k^{-1})(1-a^2)a:h_i\in
P, a\in (Q_8\times E)\setminus \zeta(Q_8\times E)\}\cup \\
&\{g(h_1-h_1^{-1})\cdots (h_k-h_k^{-1})(1-a^2)a:h_i\in P, a\in
(Q_8\times E)\setminus \zeta(Q_8\times E)\}.
\end{split}
\end{equation*}
\end{enumerate}
\end{teor}
\begin{proof}
From Theorem \ref{teorliengelQ8+}, we know that $N\simeq Q_8\times
E\times P$, where $E^2=1$, $P$ is a finite $p$-group
 and there exists $g\in G\setminus N$ such that  $G\simeq
\left<Q_8,g\right>\times E\times P$, $(g,x)=(g,y)=1$ and $g^2=x^2$.
By Lemma \ref{lema6Juhasz}, there exists $0 \neq
\widehat{P}(1-c)a\in \gamma^{m(p-1)}((FG)^+)$ for some $a\in
Q_8\times E$. In this way, $1+m(p-1)\leq \t((FG)^+)$. Furthermore,
Lemma \ref{lema5Juhasz} implies that $\tl((FG)^+)\leq \tnil(P)$.

To show (ii), consider the symmetric elements
 \linebreak
$u_i=1-a_i(1+a_i^2)+x_i$, where $x_i=a_i(h_i+a_i^2h_i^{-1})\in
\mathcal S$, $a_i\in Q_8\times E$ and $h_i\in P$. Thus,
$u_i=1+a_i(h_i-1)+a_i^3(h_i^{-1}-1)\in 1+FG\Delta(P)$. Since
$FG\Delta(P)$ is a nilpotent ideal, we get that $1+FG\Delta(P)$ is a
normal subgroup of $\mathcal{U}(FG)$ and in consequence $u_i$ is a
unit in $FG$. We will prove, by induction, that
\begin{equation}\label{eq5.24}
 (u_1,u_2,\ldots,u_n)\equiv 1+[x_1,x_2,\ldots,x_n] \pmod{
FG\Delta(P)^{n+1}}.
\end{equation}

Since $u_1^{-1}u_2^{-1}\equiv 1 \pmod{FG\Delta(P)}$, Lemma
\ref{lema5Juhasz} implies that
\begin{align*}
(u_1,u_2)&=1+u_1^{-1}u_2^{-1}[u_1,u_2]=1+(u_1^{-1}u_2^{-1}-1)[u_1,u_2]+[u_1,u_2]\\
&\equiv 1+[u_1,u_2]\pmod{FG\Delta(P)^3}.
\end{align*}
We recall that $\widehat{a}_i=1+a_i+a_i^2+a_i^3$ and $1+a_i^2$, for
each $a_i\in Q_8\times E$, are central elements of $FG$. So
\begin{align*}
[u_1,u_2]&=[1-a_1(1+a_1^2)+x_1,1-a_2(1+a_2^2)+x_2]\\
&=[x_1,x_2]+[a_1(1+a_1^2),a_2(1+a_2^2)]-[x_1,a_2(1+a_2^2)]-[a_1(1+a_1^2),x_2]\\
&=[x_1,x_2]+[\widehat{a}_1,\widehat{a}_2]-[x_1,\widehat{a}_2]-[\widehat{a}_1,x_2]\\
&=[x_1,x_2],
\end{align*}
which proves the congruence \eqref{eq5.24} when $n=2$.

Suppose that \eqref{eq5.24}, is true to $n-1$; that is
\begin{equation}
(u_1,u_2,\ldots,u_{n-1})\equiv
1+[x_1,x_2,\ldots,x_{n-1}]\pmod{FG\Delta(P)^n}.
\end{equation}
Then, Lemma \ref{lema5Juhasz} and as
$(u_1,u_2,\ldots,u_{n-1})^{-1}u_n^{-1}-1\in FG\Delta(P)$ imply
\begin{align*}
(u_1,u_2,&\ldots,u_{n})\\&=1+((u_1,u_2,\ldots,u_{n-1})^{-1}u_n^{-1}-1)[(u_1,u_2,\ldots,u_{n-1}),u_n]+[(u_1,u_2,\ldots,u_{n-1}),u_n]\\
&\equiv 1+[(u_1,u_2,\ldots,u_{n-1}),u_n]\pmod{FG\Delta(P)^{n+1}}\\
&\equiv
1+[[x_1,x_2,\ldots,x_{n-1}],1-a_n(1+a_n^2)+x_n]\pmod{FG\Delta(P)^{n+1}}\\
&\equiv
1+[x_1,x_2,\ldots,x_n]-[[x_1,x_2,\ldots,x_{n-1}],\widehat{a}_n]\pmod{FG\Delta(P)^{n+1}}\\
&\equiv 1+[x_1,x_2,\ldots,x_n] \pmod{FG\Delta(P)^{n+1}},
\end{align*}
and the statement \eqref{eq5.24} is true for all $n\geq 2$.

Let $n=\tnil(P)-1$. If $\t((FG)^+)=\tnil(P)$, then there are
$x_1,\ldots,x_n\in \mathcal{S}$ such that $[x_1,\ldots,x_n]\neq 0$.
Thus, by the congruence \eqref{eq5.24},
$\gamma_n(\mathcal{U}^+(FG))\neq 1$. So $n\leq
\cl(\mathcal{U}^+(FG))$. Moreover, we know that
$\cl(\mathcal{U}^+(FG))< \tl((FG)^+)\leq \tnil(P)=n+1$ and we get
(ii).

Assume that $P$ is powerful. Then, by Lemma \ref{lema5.3.4}, we
obtain
\begin{equation*}
x_{n,n}f_{1,\ldots,n}(1-c)a\equiv
2^nf_{1,\ldots,n}(1-c)a\pmod{FG\Delta(P)^{n+1}}.
\end{equation*}
Furthermore, if $h_i-1\in \Delta(P)^{k_i}$, then by \eqref{eq5.5}
\begin{equation*}
h_i-h_i^{-1}=(h_i-1)-(h_i^{-1}-1)\equiv 2(h_i-1)
\pmod{\Delta(P)^{k_i+1}},
\end{equation*}
thus
\begin{align*}
x_{n,n}f_{1,\ldots,n}(1-c)a&\equiv 2^n
(h_1-h_1^{-1})\cdots(h_n-h_n^{-1})(1-c)a\\
&\equiv 2^{2n}(h_1-1)\cdots(h_n-1)(1-c)a \pmod{FG\Delta(P)^{n+1}}.
\end{align*}
It is clear that, if $n<\tnil(P)$, there exist $h_1,\ldots,h_n\in P$
such that $\prod_{i=1}^n(h_i-1)\neq 0$, and then $x_{n,n}M_n\neq 0$.
Thus, $\tnil(P)\leq\t((FG)^+)$  and (iii) follows.

Finally, assume that $P$ is abelian. Let $a,b\in (Q_8\times
E)\setminus\zeta(Q_8\times E)$, and $h_1,h_2\in P$, such that
$(a,b)\neq 1$. Then
\begin{equation}\label{eq5.21}
\begin{split}
[a(h_1+a^2h_1^{-1}),b(h_2+b^2h_2^{-1})]&=(h_2+b^2h_2^{-1})(h_1+a^2h_1^{-1})[a,b]\\
&=(h_2-h_2^{-1})(h_1-h_1^{-1})(1-c)ab.
\end{split}
\end{equation}
If $\alpha\in FP$, $h\in P$, then
\begin{equation}\label{eq5.22}
\begin{split}
[\alpha(1-c)a,b(h+b^2h^{-1})]&=\alpha(1-c)(h+b^2h^{-1})[a,b]\\
&=\alpha(h-h^{-1})(1-c)^2ab=2\alpha(h-h^{-1})(1-c)ab.
\end{split}
\end{equation}

\begin{equation}\label{eq5.28}
\begin{split}
[a(h_1+a^2h_1^{-1}),bg(h_2-h_2^{-1})]&=(h_2-h_2^{-1})(h_1+a^2h_1^{-1})[a,bg]\\
&=g(h_2-h_2^{-1})(h_1+ch_1^{-1})(1-c)ab\\
&=g(h_2-h_2^{-1})(h_1-h_1^{-1})(1-c)ab,
\end{split}
\end{equation}
and
\begin{equation}\label{eq5.29}
\begin{split}
[ag(h_1-h_1^{-1}),bg(h_2-h_2^{-1})]&=(h_2-h_2^{-1})(h_1-h_1^{-1})[ag,bg]\\
&=(h_2-h_2^{-1})(h_1-h_1^{-1})g^2(1-c)ab\\
&=-(h_2-h_2^{-1})(h_1-h_1^{-1})(1-c)ab.
\end{split}
\end{equation}
The equations \eqref{eq5.22}, \eqref{eq5.28} and \eqref{eq5.29}
imply that $\gamma^2((FG)^+)=\mathcal{M}_2$. Suppose that \linebreak
$\gamma^{n-1}((FG)^+)=\mathcal{M}_{n-1}$ for some $n\geq 3$. Take
$\alpha \in FP$, $h\in P$ and $a,b\in (Q_8\times E)\setminus
\zeta(Q_8\times E)$, such that $(a,b)\neq 1$. We get the following
equalities:
\begin{equation}\label{eq5.30}
\begin{split}
[\alpha(1-c)a,gb(h-h^{-1})]&=\alpha(h-h^{-1})(1-c)[a,gb]\\
&=g\alpha(h-h^{-1})(1-c)[a,b]\\
&=g\alpha(h-h^{-1})(1-c)^2ab\\
&=2g\alpha(h-h^{-1})(1-c)ab,
\end{split}
\end{equation}
and
\begin{equation}\label{eq5.31}
\begin{split}
[g\alpha(1-c)a,gb(h-h^{-1})]&=g^2\alpha(h-h^{-1})(1-c)[a,b]\\
&=c\alpha(h-h^{-1})(1-c)[a,b]\\
&=c\alpha(h-h^{-1})(1-c)^2ab\\
&=-2\alpha(h-h^{-1})(1-c)ab.
\end{split}
\end{equation}
By substituting $f_{1,\ldots,n-1}$ for $\alpha$ in \eqref{eq5.22},
\eqref{eq5.30} and \eqref{eq5.31}, we get that \linebreak
$[\mathcal{M}_{n-1},(FG)^+]=\mathcal{M}_n$ and therefore
\begin{equation*}
\gamma^n((FG)^+)=[\gamma^{n-1}((FG)^+),(FG)^+]=[\mathcal{M}_{n-1},(FG)^+]=\mathcal{M}_n,
\end{equation*}
as we wanted to prove.
\end{proof}

\section*{Acknowledgements}
The results of this paper are part of the author's Ph.D. thesis, at
the Instituto de Ma\-te\-má\-ti\-ca e Estatística of the
Universidade de São Paulo, under the guidance of Prof. César Polcino
Milies.  This work was partially supported by CAPES and CNPq. proc.
141857/2011-0 of Brazil.

%------------------------------------------------------------
% PARA USAR BIBLIOGRAFIA CON .bib
\bibliographystyle{amsplain}
\bibliography{john}   %ARCHIVO john.bib
%------------------------------------------------------------

\end{document}